\documentclass{amsart}

\usepackage{graphicx}
\usepackage{url}
\usepackage{amsmath,amssymb,amsthm}
\usepackage{color}
\usepackage{xy}
\usepackage{tikz-cd}

\usepackage{textcomp}
\usepackage{hyperref}
\usepackage{cleveref}

\theoremstyle{plain}
\newtheorem{thm}{Theorem}[section]
\newtheorem{lem}[thm]{Lemma}
\newtheorem{prop}[thm]{Proposition}
\newtheorem{cor}[thm]{Corollary}
\newtheorem{question}[thm]{Question}

\theoremstyle{definition}

\newtheorem{nota}[thm]{Notation}

\theoremstyle{remark}
\newtheorem{Rk}[thm]{Remark}

\newcommand{\Isog}{\mathsf{I}\mathsf{s}\mathsf{o}\mathsf{g}}
\newcommand{\CM}{\mathsf{C}\mathsf{M}}
\newcommand{\fCM}{\mathfrak{C}\mathfrak{M}}
\newcommand{\fCMI}{\mathfrak{C}\mathfrak{M}\mathfrak{I}}

\newcommand{\CC}{\mathbb{C}}
\begin{document}

\title{Decidability of some 
complicated structures definable in \texorpdfstring{$\mathbb{C}(t)$}{C(t)}}
\author{Thomas Scanlon}
\email{scanlon@math.berkeley.edu}
\address{Department of Mathematics \\ University of California, Berkeley \\ Evans Hall \\ Berkeley, CA 94720-3840 \\ USA}
\dedicatory{In memory of Thanases Pheidas, 1958 -- 2023}
\begin{abstract}
Several infinite unions of algebraic sets in $\CC^n$ are 
definable in $\CC(t)$, including the set $\CM$ of $j$-invariants of 
complex elliptic curves with complex multiplication.  It has been suggested
that one could prove the undecidability of $\operatorname{Th}(\CC(t))$ by
showing that the theory of the structure $\fCM := (\CC,+,-,\cdot,0,1,\CM)$
of the field of complex numbers considered with a
unary predicate picking out $\CM$  is undecidable.  We show, using 
an effective version of the Andr\'{e}--Oort conjecture, that to the 
contrary $\operatorname{Th}(\fCM)$ is stable and decidable. 
We discuss some related structures on the complex numbers 
definable in $\CC(t)$ and how their theories may be connected
to the Zilber--Pink conjectures.
\end{abstract}

\maketitle

\section{Introduction}

It is a longstanding open problem~\cite{Rob-63} whether the first-order theory of the 
field $\CC(t)$ of rational functions in the 
single variable $t$ with coefficients from the 
field $\CC$ of complex numbers is decidable.
In various public lectures, Pheidas  described 
a strategy, which he attributed to an anonymous
reviewer of a grant proposal\footnote{After I talked about this topic some years ago, Clifton 
Ealy wrote to me saying that he had asked 
Lou van den Dries whether he may have been the 
source of this observation.  Van den Dries' answer 
on 30 June 2008 was that he recalled
having learned about the interpretation of the 
endomorphism rings (as groups) of elliptic curves 
in $\CC(t)$ from Pheidas and then noticed the 
consequence that this would give the 
definability of $\CM$.  This is not exactly a
confirmation that he was the source of the 
observation, but it is also not exactly a 
denial.}, 
for showing that this theory is undecidable based 
on interpreting a complicated structure on the 
complex numbers.

Using the fact that there are no nonconstant rational
maps from the projective line to curves of positive
genus, it is easy to see that $\CC$ is 
definable in $\CC(t)$.  For example, we have
\[ \CC = \{ a \in \CC(t) : (\exists y  \in \CC(t))~ 
a^3 + y^3 = 1 \} \text{ .} \]

The harder observation is that the set $\CM$ of 
$j$-invariants of elliptic curves with complex 
multiplication is also definable in 
$\CC(t)$.  The set $\CM$ is a countably infinite
set of algebraic integers and is known to have rich 
arithmetic.  As such, one might expect that 
the structure $\fCM := (\CC,+,-,\cdot,0,1,\CM)$
of the field of complex numbers considered with a
unary predicate picking out $\CM$ would have a 
complicated and, in particular, an undecidable theory.
Since $\fCM$ is interpretable in $\CC(t)$, if this 
expectation were verified, then one would have 
succeeded in showing that the theory of $\CC(t)$ is 
undecidable.

In this note, we observe that it follows from 
theorems 
of 
Zilber~\cite{Zilber-mt-special} on the model theory of
special varieties, of Casanovas and Ziegler~\cite{Casanovas-Ziegler-stable-theories-predicate}
on expansions of stable structures by predicates, and 
of Binyamini~\cite{Bin-effective-AO} on an
effective version of the Andr\'{e}--Oort conjecture 
for products of modular curves that the theory of 
$\fCM$ is stable and decidable.  

The considerations which give the definability of $\CM$ 
also yield the definability of several other sets which are 
naturally infinite unions of algebraic sets.  For
example, the set  \[\Isog := \{(x,y) \in \CC^2 : (\exists E)(\exists E')(\exists \psi:E \to E' \text{ an isogeny})~ j(E) = x 
\land j(E') = y \}\]
is definable in $\CC(t)$ where $j(E)$ denotes the $j$-invariant of the elliptic curve 
$E$ and an isogeny between elliptic curves is a map of algebraic groups 
with a finite kernel.  I would \emph{guess}
that the Zilber--Pink conjecture~\cite{Pink,Zilber-dim,ZP-panoramas,Zan-unlikely} 
would be relevant in the 
analysis of such structures as \(\fCMI := (\CC,+,-,\cdot,0,1,\CM,\Isog)\). 
However, I have been unsuccessful in showing both that stability of $\fCMI$ follows 
from the Zilber--Pink conjecture and that decidability of the theory of $\fCMI$ follows
from an effective form of the Zilber--Pink conjecture.   

This paper is organized as follows.  
In Section~\ref{sec:background-ec} we recall some basic 
results and constructions involving elliptic curves. 
In Section~\ref{sec:interpreting}
we explain in detail how to define $\CM$, $\Isog$, and related sets in $\CC(t)$.  Section~\ref{sec:stability}
contains the proof of the stability and decidability of $\operatorname{Th}(\fCM)$.  
 We use the Andr\'{e}--Oort conjecture proven by Pila to describe the 
quantifier-free induced structure on \(\CM\), 
a quantifier elimination due to Zilber to see that there is no more structure, 
a theorem of Casanovas and Ziegler to lift this analysis to a description of 
\(\fCM\) yielding stability, and then an effective refinement of the Andr\'{e}--Oort conjecture 
due to Binyamini to conclude decidability. Section~\ref{sec:questions}
contains some questions and observations around the connections between the Zilber--Pink conjecture and the 
complexity of the logical theory of the structures on $\CC$ interpreted in $\CC(t)$.

As noted above, I learned about the strategy for proving the undecidability of $\operatorname{Th}(\CC(t))$ from 
Thanases Pheidas. He and I talked about the relevance of the Andr\'{e}--Oort conjecture to this problem on several 
occasions, initially at the 2007 ICMS workshop on Number Theory and Computability and
most recently at the Summer 2022 MSRI program on Decidability, Definability and Computability in Number Theory. I had been 
hoping to resolve the kinds of problems described in Section~\ref{sec:questions} before releasing this 
paper, but it is clear that I have already waited too long. I had wanted to show the final version to Thanases, but
it is too late for that.  Instead, this paper is dedicated to his memory.

I have presented these ideas in various lectures, including at the 2023 Arizona Winter School.  
Further exposition of the number theoretic background is provided in my notes~\cite{Sc-AWS23} for
those lectures.

\section*{Acknowledgments}
Part of this research was performed while the author was visiting the Mathematical Sciences Research Institute (MSRI), which is supported by the National Science Foundation (Grant No. DMS-1928930).  The work has also been supported by NSF grants DMS-1800492, DMS-2201045, and DMS-2554631. 
The AI tools Gemini 2.5 Pro, Claude Sonnet 4.5, and Claude Fable 5.1 
were used only to 
proofread the \LaTeX~file but not for any part of
the mathematical research.  The author thanks the referee for an exceptionally 
careful and thorough review and for suggesting several improvements to this 
text.

\section{Some background on elliptic curves}
\label{sec:background-ec}

An elliptic curve \(E\) over a field \(K\) is 
an irreducible, projective, one-dimensional 
(necessarily commutative) algebraic 
group over \(K\).  
For more details on elliptic curves, 
consult~\cite{Husemöller} or~\cite{Silv-AEC}.  For higher-dimensional
abelian varieties, see~\cite{milneAV}. 

\begin{nota}
\label{nota:elliptic-curves-Weierstrass}
Fix a field \(K\) of characteristic not equal to \(2\) or \(3\).
We shall take our elliptic curves in 
Weierstrass form.  
For \(g_2 \in K\) and 
\(g_3 \in K\) with 
\(\Delta(g_2,g_3) := g_2^3 - 27 g_3^2 \neq 0 \)
we write \(E = E_{g_2,g_3}\) for the elliptic curve whose affine equation is 
\(y^2 = 4x^3 - g_2 x - g_3 \) and which has
one point at infinity.  As explained in our
references, there is an algebraic group structure on 
\(E\) in which the point at infinity is the identity and 
addition is given by the secant-tangent method.
\end{nota}

In practice, we will only work with constant \(g_2\) and \(g_3\),
that is, we take \(K = \CC\).

\begin{nota}
\label{nota:j-alg}
 The algebraic \(j\)-invariant
of \(E\) is \[j^{\operatorname{alg}}(g_2,g_3) := \frac{12^3 g_2^3}{\Delta} \] and it is known that over an algebraically closed field
\(E_{g_2,g_3} \cong E_{g_2',g_3'}\) as algebraic groups
if and only if 
\(j^{\operatorname{alg}}(g_2,g_3) = j^{\operatorname{alg}}(g_2',g_3')\).

Using this invariance under isomorphism, we write \(j(E)\) 
for \(j^{\operatorname{alg}}(g_2,g_3)\) if 
\(E\) is an elliptic curve which is isomorphic to the elliptic curve 
\(E_{g_2,g_3}\).
\end{nota}

We may understand elliptic curves over \(\CC\)
via complex analysis. See~\cite[Chapter VI]{Silv-AEC} and~\cite[Chapter 11]{Husemöller} for some of this 
theory. 

\begin{nota}
\label{nota:fractional-linear}
Let \(\operatorname{GL}_2(\mathbb{R})^+\) denote the 
group of \(2 \times 2\) matrices with real entries and 
positive determinant.  There is an action of \(\operatorname{GL}_2(\mathbb{R})^+\) on 
\(\mathfrak{h}  := \{ z \in \CC : \operatorname{Im}(z) > 0 \}\) via the rule 
\[\left( \begin{matrix} a & b \\ c & d \end{matrix} \right) \cdot \tau := \frac{a \tau + b}{c \tau + d} \text{ .}\]
\end{nota}

If \(\tau \in \mathfrak{h}\), then the complex Lie group 
\(T_\tau := \CC / (\mathbb{Z} + \mathbb{Z} \tau)\) 
is isomorphic, as a complex Lie group, to 
(the analytification
of) an elliptic curve.  This identification is achieved
through the use of Eisenstein series and the Weierstrass
\(\wp\)-function.   The exact formulae are not really 
necessary for our discussion, but for the sake of culture we 
recall them.

\begin{nota}
\label{nota:wp}
The Weierstrass \(\wp\)-function is a meromorphic map
\(\wp:\CC \times \mathfrak{h} \to \CC\) defined by 
\[\wp(z,\tau) := \frac{1}{z^2} + \sum_{(n,m) \in 
\mathbb{Z}^2 \smallsetminus \{(0,0)\}} \left( \frac{1}{(z - n - m\tau)^2} - \frac{1}{(n+m\tau)^2} \right) \text{ .}\]
We write \(\wp'(z,\tau)\) for the derivative of \(\wp\) 
with respect to \(z\).  
\end{nota}

\begin{nota}
\label{nota:Eisenstein}
For \(k \in \mathbb{Z}_+\) the Eisenstein function 
\(E_{2k}:\mathfrak{h} \to \CC\) is defined by 
\[E_{2k}(\tau) := \sum_{(n,m) \in 
\mathbb{Z}^2 \smallsetminus \{(0,0)\}} \frac{1}{(n+m \tau)^{2k}} \text{ .}\]

We set \(g_2(\tau) := 60 E_4(\tau)\), 
\(g_3(\tau) := 140 E_6(\tau)\), and define Klein's 
\(j\)-function by 
\(j(\tau) := j^{\operatorname{alg}}(g_2(\tau),g_3(\tau))\).  
\end{nota}

With Notation~\ref{nota:wp} and~\ref{nota:Eisenstein}
in place, we may recall some key facts about these functions.

The complex torus \(T_\tau\) is isomorphic to (the 
analytification of) the elliptic curve 
\(E_{g_2(\tau),g_3(\tau)}\) via the map 
\(z \mapsto (\wp(z,\tau),\wp'(z,\tau))\).  

A map of complex Lie groups between 
two one-dimensional complex tori 
\(\psi:T_\tau \to T_{\tau'}\) is given by 
multiplication by some complex number \(\lambda\) for 
which \( \lambda (\mathbb{Z} + \mathbb{Z} \tau) \subseteq
(\mathbb{Z} + \mathbb{Z} \tau') \).  We say that the 
complex torus \(T_\tau\) has complex multiplication 
if and only if there is some 
\(\lambda \in \CC \smallsetminus \mathbb{Z}\) 
for which \(\lambda (\mathbb{Z} + \mathbb{Z} \tau) 
\subseteq (\mathbb{Z} + \mathbb{Z} \tau)\).  It is 
an easy computation to check that \(T_\tau\) has 
complex multiplication if and only if 
\(\mathbb{Z}[\tau]\) is a free \(\mathbb{Z}\)-module of rank two.  
We say 
that an elliptic curve \(E\) over \(\CC\) has 
complex multiplication if it (or, really, its analytification)
is a complex torus with complex multiplication.  
Alternatively, we may detect that the elliptic curve
has complex multiplication through its endomorphism 
ring \(\operatorname{End}(E)\), the set of maps of 
algebraic groups \(\psi:E \to E\) with addition defined 
by \((\psi + \phi)(x) = (\psi(x) +_E  \phi(x))\) where
the second addition is in the sense of the algebraic 
group structure on \(E\) and multiplication defined by 
composition of functions.   For an elliptic 
curve \(E\) over the complex numbers, either 
\(\operatorname{End}(E) \cong \mathbb{Z}\) or 
\(\operatorname{End}(E)\) is an order in a quadratic 
imaginary field, and in particular has rank two as  
an abelian group.  The elliptic curve \(E\) has 
complex multiplication if and only if we are in this
second case.

It is known that \(j:\mathfrak{h} \to \CC\) is surjective. 
Moreover, it follows from the properties of the 
algebraic \(j\)-function that \(j(\tau) = j(\tau')\) if
and only if \(T_\tau \cong T_{\tau'}\) if and only 
if \(E_{g_2(\tau),g_3(\tau)} \cong E_{g_2(\tau'),g_3(\tau')}\)
where the first isomorphism is taken in the category of 
complex Lie groups and the second is in the category of 
algebraic groups. 

This equivalence may be described via the action of 
\(\operatorname{GL}_2(\mathbb{R})^+\) on \(\mathfrak{h}\)
as \(j(\tau) = j(\tau')\) if and only if there is some 
\(\gamma \in \operatorname{SL}_2(\mathbb{Z})\) with
\(\tau' = \gamma \cdot \tau\).  

Taking 
\(\gamma \in \operatorname{GL}_2(\mathbb{Q})^+ := 
\operatorname{GL}_2(\mathbb{Q}) \cap \operatorname{GL}_2(\mathbb{R})^+\) instead, 
we find that the complex tori 
\(T_\tau\) and \(T_{\gamma \cdot \tau}\) are 
\emph{isogenous}, even by cyclic isogenies.
That is, there is a map of complex tori
\(T_\tau \to T_{\gamma \cdot \tau}\) whose kernel
is a finite cyclic group.   This relation is 
expressed by the modular polynomials.

\begin{nota}
\label{nota:modular-poly}
For each positive integer 
\(\ell \in \mathbb{Z}_+\), we let 
\(\Phi_\ell(x,y) \in \mathbb{Z}[x,y]\) be the 
\(\ell^\text{th}\) modular polynomial~\cite[page 335]{Knapp}.  
This polynomial \(\Phi_\ell\) is 
characterized by the properties that it is 
monic in \(x\) and for any pair of 
complex numbers \((j_1,j_2)\) we have 
\(\Phi_\ell(j_1,j_2) = 0\) if and only if there
are elliptic curves \(E_1\) and \(E_2\)
with \(j(E_1) = j_1\), \(j(E_2) = j_2\), and
a map of algebraic groups 
\(\psi:E_1 \to E_2\) with kernel isomorphic 
to \(\mathbb{Z}/\ell\mathbb{Z}\).  
It is a general fact that if two 
elliptic curves \(E_1\) and \(E_2\) are
isogenous, meaning that there is a map of 
algebraic groups \(\psi:E_1 \to E_2\) with 
finite kernel, then \(E_1\) has complex
multiplication if and only if \(E_2\) 
does~\cite[Exercise VI.6.9]{Silv-AEC}.

For \(\gamma \in \operatorname{GL}_2(\mathbb{Q})^+\)
there is a number \(\ell(\gamma) \in \mathbb{Z}_+\)
so that \(\Phi_{\ell(\gamma)}(j(\tau),j(\gamma \cdot \tau)) \equiv 0\).  That is, the elliptic curves 
\(E_{g_2(\tau),g_3(\tau)}\) and \(E_{g_2(\gamma \cdot \tau),g_3(\gamma \cdot \tau)}\) are isogenous by 
maps \(\psi:E_{g_2(\tau),g_3(\tau)} \to E_{g_2(\gamma \cdot \tau),g_3(\gamma \cdot \tau)}\) having kernel 
cyclic of order \(\ell(\gamma)\).   Moreover, for 
each \(\ell\) there are 
\(\gamma \in \operatorname{GL}_2(\mathbb{Q})^+\) with 
\(\ell(\gamma) = \ell\).

\end{nota}

\section{Some structures interpreted in \texorpdfstring{$\CC(t)$}{C(t)}}
\label{sec:interpreting}

In this section we explain in detail how 
$\CM$, $\Isog$, and related structures are 
defined in $\CC(t)$.  As noted in the 
introduction, these constructions are not due to me, but 
as they are not documented in the literature, they are 
included here.

We start by recording the observation (which
already appears in~\cite{Rob-63}) that the 
field of constants is definable in 
\(\CC(t)\).

\begin{lem}
\label{lem:definable-constants}
The field of complex numbers \(\CC\) is a 
definable subset of \(\CC(t)\).
\end{lem}
\begin{proof}
See~\cite[Section IV.2]{Hartshorne-AG} 
for the necessary results on the genus of 
algebraic curves.

It suffices to consider any two-variable polynomial 
\(f(x,y) \in \mathbb{Z}[x,y]\) for which the 
equation \(f(x,y) = 0\) defines the affine part of a 
smooth curve \(X\) of genus at least one and for which 
the one-variable polynomial \(f(a,y)\) is nonconstant 
for every choice of \(a \in \CC\).  Then the 
formula \(\phi(x) := (\exists y) f(x,y) = 0 \) defines
\(\CC\).  Indeed, if \(a \in \CC\),
then because \(f(a,y) \in \CC[y]\) is a nonconstant
polynomial, it has a root.  On the other hand, if 
\(a \in \CC(t)\) and \(b \in \CC(t)\) witnesses 
that \(\phi(a)\) holds, then \((a,b)\) defines 
a rational map \(\mathbb{P}^1_\CC \dashrightarrow X\). 
As the genus of \(X\) is at least one, this map is 
necessarily constant, which means in particular
that \(a\) is constant.

For a concrete choice of \(f\), we may take 
\(f(x,y) = x^3 + y^3 - 1\).
\end{proof}

There is a standard construction to interpret a 
finite algebraic extension \(L\) of the field 
\(K\) in \(K\).  Indeed, for any fixed positive 
integer \(d\) the class of all extensions of 
\(K\) of degree \(d\) is uniformly interpretable.
Let us formalize this fact with the following 
proposition.

\begin{prop}
    \label{prop:interpret-finite-extension}
    For each positive integer \(d\) the class of field extensions of degree \(d\) is uniformly interpretable relative to the theory of fields.  That is, there is a formula \(\theta(\mu)\) in the \(d^3\) variables \(\mu\), a formula \(\xi\) in \(d\) variables, and definable functions \(\otimes_\mu\) and \(\oplus_\mu\) taking \(2d\) arguments and returning a \(d\)-tuple, depending on the parameters \(\mu\) so that for any field \(K\) and any \(d \times d \times d\)-array \(\mu\) with
    entries from \(K\) such that \(K \models \theta(\mu) \), we have that \( L_\mu := (K^d,\oplus_\mu,\otimes_\mu)\) is a field and \(\xi(K) \subseteq L_\mu\) picks out a subfield isomorphic to \(K\) for which \([L_\mu:K]=d\). Moreover, for every extension \(L\) of \(K\) of degree \(d\) there is some choice of parameters \(\mu\) for which \(L_\mu \cong L\) over the identification of \(K\) with \(\xi(K)\).     
\end{prop}
\begin{proof}
While this construction is standard, allow us to recall some of the details.  For an extension \(L\) of the field \(K\) of degree \(d\), we may fix some \(K\)-basis \(\{e_1, \ldots, e_d \}\) of \(L\).   Making a change of basis if need be, we may assume that \(e_1 = 1\).  
    
Let \(\mu\) be the \(d \times d \times d\) array of coefficients for which \[e_i e_j = \sum_{k=1}^d \mu_{i,j,k} e_k \text{ .}\] Note that the coefficients \(\mu_{i,j,k}\) are uniquely determined as \(\{e_1, \ldots, e_d \}\) is a \(K\)-basis.  Since \(e_i e_j = e_j e_i\), we could get away with recording these values only for \(i \leq j\), but there is no harm in writing too much. 

Let \(\xi(x_1, \ldots,x_d)\) be the formula \(\bigwedge_{i=2}^d x_i = 0 \).  Define {\small  \[(x_1,\ldots,x_d) \oplus_\mu (y_1,\ldots,y_d) := (x_1 + y_1, 
\ldots, x_k+y_k, \ldots, x_d + y_d)\]} and {\small
\[(x_1, \ldots, x_d) \otimes_\mu (y_1, \ldots, y_d) := (\sum_{1 \leq i, j \leq d} x_i y_j \mu_{i,j,1}, \ldots, \sum_{1 \leq i, j \leq d} x_i y_j \mu_{i,j,k}, \ldots, \sum_{1 \leq i, j \leq d} x_i y_j \mu_{i,j,d}) \text{ .} \]}

Then, \(L \cong L_\mu = (K^d, \oplus_\mu, \otimes_\mu)\) and \(\xi(K)\) identifies with the image of \(K\) in \(L\) under the usual embedding of \(K\)-algebras.  

In general, the formula \(\theta(\mu)\) asserts that \(L_\mu\) is a field with multiplicative identity \((1,0,\ldots,0)\) by relativizing the finitely many field axioms to \(K\). 
\end{proof}

Let us use the construction of 
Proposition~\ref{prop:interpret-finite-extension}  to encode the
function fields of the elliptic curves 
\(E_{g_2,g_3}\) as \(g_2\) and \(g_3\) range through 
\(\CC\) with \(\Delta \neq 0 \).

\begin{nota}
\label{notation:ff-elliptic}
For \(g_2 \in \CC\) and \(g_3 \in \CC\) 
with \(\Delta(g_2,g_3) \neq 0\), let 
\(\mu\) be the array describing multiplication on the 
field \(L_\mu := L^{g_2,g_3} := \CC(t)[y]/(y^2 - 4t^3 + g_2 t + g_3)\) 
 as a quadratic extension of \(K = \CC(t)\) 
relative to the basis \((e_1, e_2) = (1,y)\).  That is, 
\(\mu_{1,1,1} = 1\), \(\mu_{1,1,2} = 0 \), \(\mu_{1,2,1} = 0\), 
\(\mu_{1,2,2} = 1\), \(\mu_{2,1,1} = 0\), \(\mu_{2,1,2} = 1\), 
\(\mu_{2,2,1} = 4t^3 - g_2 t - g_3\), and \(\mu_{2,2,2} = 0\).
\end{nota}

\begin{prop}
\label{prop:rank-maps-elliptic}
For \(r = 0\), \(1\), and \(2\) there is a formula \(\rho_r(g_2,g_3,g_2',g_3')\) of
the four variables \(g_2\), \(g_3\), \(g_2'\), and \(g_3'\) 
so that 
\(\CC(t) \models \rho_r(g_2,g_3,g_2',g_3')\) if and only if 
\begin{itemize}
    \item \(g_2\), \(g_3\), \(g_2'\), and \(g_3'\) all 
    belong to \(\CC\), 
    \item \(\Delta(g_2,g_3) \neq 0\) and \(\Delta(g_2',g_3') \neq 0\), and 
    \item \( \operatorname{rk} \operatorname{Hom}(E_{g_2,g_3},E_{g_2',g_3'}) = r \).
\end{itemize}
\end{prop}
\begin{proof}
By Lemma~\ref{lem:definable-constants} the first two conditions may be 
expressed with a formula.  Using the fact that every rational map from 
one elliptic curve to another takes the form of a map of algebraic groups followed by
a translation~\cite[Corollary 3.7]{Milne-AV}, we see that the interpretable group 
\(E_{g_2',g_3'}(L^{g_2,g_3}) = E_{g_2',g_3'}(\CC(E_{g_2,g_3}))\) fits into the following exact 
sequence:
\[ \begin{tikzcd}  0 \ar[r] & E_{g_2',g_3'}(\CC) \ar[r] & E_{g_2',g_3'}(L^{g_2,g_3}) \ar[r] & \operatorname{Hom}(E_{g_2,g_3},E_{g_2',g_3'}) \ar[r] & 0 \end{tikzcd} \]

That is, as a group, we may identify \(\operatorname{Hom}(E_{g_2,g_3},E_{g_2',g_3'})\) 
with the interpretable group \(E_{g_2',g_3'}(L^{g_2,g_3})/E_{g_2',g_3'}(\CC)\).  
On general grounds (see~\cite[Theorem VI.6.1]{Silv-AEC}), 
this group is isomorphic to one of 
\begin{itemize} 
\item \(0\), if \(E_{g_2,g_3}\) and \(E_{g_2',g_3'}\) are not isogenous,
\item \(\mathbb{Z}\), if \(E_{g_2,g_3}\) and \(E_{g_2',g_3'}\) are isogenous and 
\(E_{g_2',g_3'}\) does not have complex multiplication, or  
\item  \(\mathbb{Z}^2\), if \(E_{g_2,g_3}\) and \(E_{g_2',g_3'}\) are isogenous and 
\(E_{g_2',g_3'}\) has complex multiplication.
\end{itemize}
We may recognize the
first case that \(\operatorname{Hom}(E_{g_2,g_3},E_{g_2',g_3'}) = 0\), 
with the formula expressing that \(E_{g_2',g_3'}(L^{g_2,g_3}) = E_{g_2',g_3'}(\CC)\), 
giving the formula \(\rho_0(g_2,g_3,g_2',g_3')\).    To recognize 
the second case we use the formula expressing that there is some 
\(P \in E_{g_2',g_3'}(L^{g_2,g_3}) \smallsetminus E_{g_2',g_3'}(\CC)\) and that for 
all \(Q \in E_{g_2',g_3'}(L^{g_2,g_3})\) there is some \(R \in E_{g_2',g_3'}(L^{g_2,g_3})\) 
so that either \([2] R = Q\) or \([2] R = Q + P\), where 
multiplication by two and addition are taken in the sense of the 
group law on \(E_{g_2',g_3'}\).  This gives the formula 
\(\rho_1(g_2,g_3,g_2',g_3')\).  Finally, \(\rho_2(g_2,g_3,g_2',g_3')\) may 
be taken to be the formula asserting that there are \(P\) and \(Q\) in 
\(E_{g_2',g_3'}(L^{g_2,g_3})\) so that for all \(R\) none of the equations
\([2] R = P\), \([2] R = Q\), or \([2]R = P + Q\) hold.
\end{proof}

\begin{cor}
\label{cor:def-cm-isog}
Each of the sets \(\CM\) and \(\Isog\) is definable in 
\(\CC(t)\). Hence, each of the structures \(\fCM = (\CC,+,-,\cdot,0,1,\CM)\) and 
\(\fCMI = (\CC,+,-,\cdot,0,1,\CM,\Isog)\) is interpretable in 
\(\CC(t)\).
\end{cor}
\begin{proof}
The set \(\CM\) is defined by the formula \[\exists u \exists v (\rho_2(u,v,u,v) \land x = j^{\operatorname{alg}}(u,v))\]  
and \(\Isog\) is defined by \[\exists u \exists v \exists w \exists z (x = j^{\operatorname{alg}}(u,v) \land y = j^{\operatorname{alg}}(w,z) \land (\rho_1(u,v,w,z) \vee \rho_2(u,v,w,z))) \text{ .}\]
\end{proof}

\section{Stability and decidability of \texorpdfstring{$\fCM$}{CM}}
\label{sec:stability}

The structure \(\fCM\) is not definable in 
\(\CC\) considered as a field.  Indeed, the algebraically 
closed field \(\CC\) is strongly minimal but the set 
\(\CM\) is a countably infinite set of algebraic integers, so
that both it and its complement are infinite.  
One might expect that some of the arithmetic of the algebraic 
integers could be recovered in \(\fCM\).  As 
the theory of the ring of all algebraic integers is itself 
decidable~\cite{vdD-alg-int}, this would not yield undecidability 
of \(\fCM\), but it would imply instability.  In this section, 
we observe that it follows from the 
Andr\'{e}--Oort conjecture for products of modular 
curves~\cite{Pila-AO} and a theorem of 
Casanovas 
and Ziegler~\cite{Casanovas-Ziegler-stable-theories-predicate}
on expansions of stable theories by predicates
that \(\fCM\) is stable and then using another 
part of the work of Casanovas and Ziegler 
and an effective version of the Andr\'{e}--Oort conjecture
due to Binyamini~\cite{Bin-effective-AO}, that 
the theory of \(\fCM\) is decidable.

\begin{nota}
\label{nota:induced}
Recall that for any \(\mathcal{L}\)-structure \(\mathfrak{M}\)
and a nonempty subset \(A \subseteq M^n\) of the \(n^\text{th}\)
Cartesian power of the universe \(M\) of \(\mathfrak{M}\) 
(for some natural number \(n\)), the induced structure, 
\(A^{\operatorname{ind}}\), on \(A\) is the structure having 
universe \(A\) and an \(m\)-ary 
relation \(R_\phi\) 
for each \(\mathcal{L}\)-formula \(\phi\) in the 
variables \(x_{i,j}\) for \(1 \leq i \leq m\) and
\(1 \leq j \leq n\) to be interpreted as 
\(R_\phi^{A^{\operatorname{ind}}} = A^m \cap \phi(\mathfrak{M})\).  
We refer to this associated relational language as 
\(\mathcal{L}^{\operatorname{ind}}\).  Strictly speaking, we should 
indicate the arity \(n\) in the notation for the name of this
language, but we omit it.
\end{nota}

\begin{nota}
\label{nota:L-mod}
Let \(\mathcal{L}^{\operatorname{mod}}\) be the relational language having 
an \(n\)-ary relation symbol \(R_X\) for each irreducible component \(X\)
of an 
embedded affine algebraic subvariety 
of affine \(n\)-space \(\mathbb{A}^n_{\CC}\)
defined by a system of equations of the form \(\Phi_{\ell_k}(x_{i_k},x_{j_k}) = 0\) for
some sequence \((i_1,j_1,\ell_1), \ldots, (i_m,j_m,\ell_m)\) where 
\(1 \leq i_k \leq j_k \leq n\) and \(\ell_k \in \mathbb{Z}_+\) for \(1 \leq k \leq m\). 
We call such \(X\) \emph{special varieties}.  
A component of a variety for which
we further allow equations of 
the form \(x_i = c\) for 
\(c \in \CC\) is called 
\emph{weakly special}.

Let \(\CC^{\operatorname{mod}}\) be the \(\mathcal{L}^{\operatorname{mod}}\)-structure with universe
\(\CC\) in which for each special variety
\(X\) we interpret
\(R_X^{\CC^{\operatorname{mod}}}\) as \(X(\CC)\).  Let 
\(\CM^{\operatorname{mod}}\) be the substructure of \(\CC^{\operatorname{mod}}\) with universe \(\CM\).
\end{nota}

\begin{Rk}
An equation of the form \(\Phi_\ell(x,x)=0\) with 
\(\ell > 1\) defines a finite set of CM-points and more 
generally a pair of equations \(\Phi_\ell(x,y) = 0\) and 
\(\Phi_{\ell'}(x,y) = 0\) with \(\ell \neq \ell'\) will 
define a finite set of pairs of CM-points.  It follows
from these observations that if \(X \subseteq \mathbb{A}^n_\CC\)
is a special variety, then for each projection 
\(\rho_j:\mathbb{A}^n_\CC \to \mathbb{A}^1_\CC\) given 
by \(\rho_j(x_1, \ldots, x_n) = x_j\), either 
the image of \(X\) under \(\rho_j\) is infinite, or 
\(\rho_j(X) = \{\xi\}\) for some CM-point \(\xi\).
\end{Rk}

\begin{prop}
\label{prop:CM-mod-elementary}
The theory of the structure \(\CM^{\operatorname{mod}}\) has 
quantifier elimination, is strongly minimal, and, hence, 
a fortiori, stable, and is decidable.
\end{prop}

In Section 3 of the original version of the 
present paper, available as a preprint at~\cite{Sc-preprint-decidability},
a more hands-on proof of Proposition~\ref{prop:CM-mod-elementary} is given.  
As most of this proposition is essentially proven in~\cite{Zilber-mt-special},
we follow this route.  However,
we will need to explain how to  correct some minor 
misstatements and to extract the present proposition from the 
more abstract results shown there.    

\begin{proof}
Zilber works in a more general context where 
\(\mathbb{X}\) is a complex algebraic variety admitting an analytic
covering \(p:\mathbb{U} \to \mathbb{X}(\CC)\) where
\(\mathbb{U} \subseteq \CC^n\) is a semi-algebraic set in 
\(\CC^n\) for some \(n \in \mathbb{Z}_+\).  He defines the class of 
weakly special varieties by a bi-algebraicity condition:  \(Y \subseteq \mathbb{X}^m\) is weakly special if \(Y\) is an irreducible 
subvariety for which there is an irreducible complex analytic 
set \(\mathbb{Y} \subseteq \mathbb{U}^m\) which is also 
semi-algebraic for which \(Y(\CC) = p(\mathbb{Y})\).
Here, we write \(p\) also for the function defined on \(m\)-tuples by the
coordinatewise action of \(p\).  He considers the 
structure \(\mathbb{X}_{w}\) having universe \(\mathbb{X}\) 
and basic \(m\)-ary 
predicates for each weakly special subvariety of \(\mathbb{X}^m\).  
Theorem 4.11(b) of~\cite{Zilber-mt-special} says that 
the theory of \(\mathbb{X}_{w}\) has quantifier elimination.  
It is known in the case we are considering, where  
\(\mathbb{X} = \mathbb{A}^1_\CC\), \(\mathbb{U} = \mathfrak{h}\), and
\(p = j\), that the weakly special varieties as defined
by Zilber are the same as what we call weakly special varieties in 
Notation~\ref{nota:L-mod}. 
For instance, this follows from the modular Ax-Lindemann-Weierstrass
theorem of Pila~\cite{Pila-AO} and is part of the general 
theory of bi-algebraicity~\cite{KUY-bi-alg}.

In~\cite[Section 6]{Zilber-mt-special} Zilber
defines a reduct \(\mathbb{X}_s\) in which we have 
distinguished predicates only for 
\emph{special} subvarieties.

The 
conditions as given in~\cite[Definition 6.1(iii)]{Zilber-mt-special} 
for the special varieties are 
not quite right:  instead, the correct definition is that 
\(S \subseteq \mathbb{X}^n\) is \emph{special} if \(S\) is 
weakly special, defined over \(\mathbb{Q}^{\operatorname{alg}}\), and 
for some weakly co-special  \(\widetilde{S}\) corresponding to 
\(S\) the Zariski closure \(\widetilde{S}^{\operatorname{Zar}}\)
is defined over \(\mathbb{Q}^{\operatorname{alg}}\) (whereas the 
original definition required  \(\widetilde{S}^{\operatorname{Zar}}\)
to be defined over \(\mathbb{Q}\)).   That we need to 
correct the definition may be seen from \cite[Example 6.4]{Zilber-mt-special} 
where it is asserted that the special subvarieties of Shimura varieties as 
usually defined match the new definition via bi-algebraicity, which is incorrect.  
For example, regarding the affine line as the coarse moduli space of elliptic 
curves via \(j:\mathfrak{h} \to \mathbb{A}^1(\mathbb{C})\), the special points, 
i.e., the CM-points, are the images of quadratic imaginary points, not 
rational points. 

As defined 
on the first page of~\cite{UA-char}, a 
special subvariety of a Shimura variety 
is a weakly special variety containing a special point. 
From~\cite[Theorem 1.2]{UA-char}, we see that the weakly 
special subvarieties of Shimura varieties are precisely the 
weakly special subvarieties in Zilber's sense. As 
expressed in~\cite[Theorem 4.18]{KUY-bi-alg}, 
the bi-algebraic points in Shimura varieties of abelian 
type, which include products of modular curves, are 
exactly the special points in the Shimura theoretic sense 
of having a torus as a Mumford-Tate group.  In our case, this is 
a celebrated theorem of Schneider~\cite{Schneider} and
in general is due to Cohen, Shiga, and Wolfart~\cite{Cohen,ShWo}. 
That the special subvarieties of Shimura varieties of abelian
type in our sense are the same as what appears in the 
literature may be seen from~\cite[Definition 2.1]{UA-char}.  

Zilber offers two conjectures, Conjecture A saying 
that every weakly special variety may be obtained as a fiber of
a special variety and Conjecture B saying that for every 
 special variety \(Y \subseteq \mathbb{X}^m\) the set 
 \(S^m \cap Y(\mathbb{Q}^{\operatorname{alg}})\) is Zariski dense in 
 \(Y\).  He notes that if
 Conjectures A and B hold for a given
 triple \((\mathbb{X}, \mathbb{U},p)\), then one may 
 deduce Conjecture AB that \(S^{\operatorname{ind}} \preceq 
 \mathbb{X}_s\).  (Here, ``\(S^{\operatorname{ind}}\)'' is the 
 structure induced on \(S\) from \(\mathbb{X}_s\).  By the Andr\'{e}--Oort 
 conjecture~\cite{AO-general}, this is the same as the structure induced on 
 \(S\) from the field structure.)
 Zilber notes that the main theorem 
 of~\cite{UA-char} shows that Conjecture A holds for 
 Shimura varieties and that Conjecture B is also known for 
 Shimura varieties.  His observations are correct if 
 ``special point'' is understood in the usual Shimura theoretic sense,
 which in the case of Shimura varieties of abelian type, including
 in particular products of modular curves, agrees with our corrected 
 definition of special point. Thus, with the corrected 
 definition of \emph{special variety}, these observations 
 are correct and Shimura varieties of abelian type satisfy
 Conjecture AB.

 Let us expand on 
 Zilber's observations. The key to quantifier elimination for 
 \(\mathbb{X}\) in the language where every weakly special 
 variety is named by a predicate 
 is~\cite[Proposition 4.8]{Zilber-mt-special}.  The proof of that 
 proposition
 goes through with ``special'' replacing the word ``weakly 
 special'' as the relevant Boolean operations preserve
 bi-algebraicity and the fiber of a special variety over a 
 special point is itself special.  The conclusion that 
 \(S^{\operatorname{ind}} \preceq 
 \mathbb{X}_s\) follows from this quantifier elimination result 
 and the fact that every special variety contains a Zariski 
 dense set of special points:  by quantifier elimination 
 it suffices to show that if \(Y \subseteq \mathbb{X}^{n+1}\)
 is a special variety and \(Z_1, \ldots, Z_m \subsetneq Y\) are 
 special subvarieties of \(Y\) and \(a \in \mathbb{X}^n\) is a 
 special point for which \((Y \smallsetminus (Z_1 \cup \cdots 
 \cup Z_m))_a = Y_a \smallsetminus ((Z_1)_a \cup \cdots \cup (Z_m)_a) 
 \neq \varnothing\), then there is a special point in 
 \(Y_a \smallsetminus ((Z_1)_a \cup \dots \cup (Z_m)_a)\).  Since
 the fiber \(Y_a\) is itself a special variety (possibly reducible), 
 the special points are Zariski dense in it.  Thus, any nonempty
 component of \(Y_a \smallsetminus ((Z_1)_a \cup \dots \cup (Z_m)_a)\)
 contains a special point.

 In our case, the set of special points \(S\) is \(\CM\) 
 by Schneider's theorem~\cite{Schneider}. Thus, 
 using Pila's  
 characterization of the weakly special varieties we see that
the special varieties as defined by Zilber's corrected strong 
bi-algebraicity condition are the same as our special varieties.  
Thus, \(\CM^{\operatorname{mod}} \preceq (\mathbb{A}^1)_s\), which is 
strongly minimal (being a reduct of the complex field),
has quantifier elimination, and is decidable (because the theory of the 
complex field is decidable and 
\(\mathcal{L}^{\operatorname{mod}}\) is a computable sublanguage).  
\end{proof}

With the next result we see that \(\CM^{\operatorname{mod}}\) is 
bi-definable with the structure
induced on \(\CM\) from the language of rings.

\begin{prop}
\label{prop:bi-definable}
The structures \(\CM^{\operatorname{mod}}\) and \(\CM^{\operatorname{ind}}\), where here we induce the structure 
from \(\CC\) in the language of rings, are bi-definable. 
\end{prop}

\begin{proof}
This is the content of the Andr\'{e}--Oort conjecture for 
products of modular curves~\cite[Theorem 1.1*]{Pila-AO}.  

Indeed, by quantifier elimination for algebraically 
closed fields, for each formula \(\phi\)
in the language of rings with free variables 
amongst \(\{x_1, \ldots, x_n\}\) we may write 
\[\phi(\CC) = \bigcup_{X \in \mathcal{X}} \left(X(\CC)  \smallsetminus 
\bigcup_{Y \in \mathcal{Y}_X} Y(\CC) \right) \] 
where \(\mathcal{X}\) is a finite set of irreducible varieties 
and for each \(X \in \mathcal{X}\) the set \(\mathcal{Y}_X\) is a 
finite set of proper closed irreducible subvarieties of \(X\).

For \(X \in \mathcal{X}\) and \(Y \in \mathcal{Y}_X\) define 
\(\widetilde{X} := \overline{X(\CC) \cap \CM^n}^{\operatorname{Zariski}}\) and \(\widetilde{Y} := \overline{Y(\CC) \cap \CM^n}^{\operatorname{Zariski}}\).

Decompose these 
as \(\widetilde{X} = \bigcup_{Z \in \mathcal{Z}_X} Z\) and then
for each \(Z \in \mathcal{Z}_X\) write 
\(Z \cap \bigcup_{Y \in \mathcal{Y}_X} \widetilde{Y} \) as
\(\bigcup_{W \in \mathcal{W}_Z} W\) where each \(W\) is irreducible.  
By the Andr\'{e}--Oort 
conjecture, each \(Z\) and \(W\) is a special variety.  That is, 
\(R_{Z}\) and \(R_{W}\)
are predicates in \(\mathcal{L}^{\operatorname{mod}}\).
Using the fact that \(\widetilde{X} \subseteq X\) and 
\(\widetilde{Y} \subseteq Y\), we see that 
\(\widetilde{X}(\CC) \cap \CM^n = X(\CC) \cap \CM^n\), 
and likewise for \(\widetilde{Y}\) and \(Y\) in place
of \(\widetilde{X}\) and \(X\).
Putting this all together we have 

\begin{align*}     
R_\phi(\CM^{\operatorname{ind}}) &= \phi(\CC) \cap \CM^n \\
            &=  \bigcup_{X \in \mathcal{X}} \left(X(\CC)  \smallsetminus 
\bigcup_{Y \in \mathcal{Y}_X} Y(\CC) \right)\cap \CM^n \\
            &=  \bigcup_{X \in \mathcal{X}} \left(\widetilde{X}(\CC)  \smallsetminus 
\bigcup_{Y \in \mathcal{Y}_X} \widetilde{Y}(\CC) \right) \cap \CM^n \\
            &=  \bigcup_{X \in \mathcal{X}} \bigcup_{Z \in \mathcal{Z}_X} \left(Z(\CC)  \smallsetminus 
\bigcup_{W \in \mathcal{W}_Z} W(\CC) \right) \cap \CM^n \\
        &=  \bigcup_{X \in \mathcal{X}} \bigcup_{Z \in \mathcal{Z}_X} \left(Z(\CC) \cap 
        \CM^n \smallsetminus 
\bigcup_{W \in \mathcal{W}_Z} (W(\CC) \cap \CM^n )\right)  \\
        &= \widehat{\phi}(\CM^{\operatorname{mod}})
\end{align*}

where  
\begin{align*}
\widehat{\phi} &= \bigvee_{X \in \mathcal{X}} \bigvee_{Z \in \mathcal{Z}_X} \left(R_Z 
         \land \neg  
\bigvee_{W \in \mathcal{W}_Z} R_W \right) 
\end{align*}
is a quantifier-free formula in \(\mathcal{L}^{\operatorname{mod}}\). 

Thus, every basic predicate in \(\mathcal{L}^{\operatorname{ind}}\) is 
interpreted in \(\CM^{\operatorname{ind}}\) by a quantifier-free 
\(\mathcal{L}^{\operatorname{mod}}\)-formula.  Hence, 
every \(\mathcal{L}^{\operatorname{ind}}\)-definable set in 
\(\CM^{\operatorname{ind}}\) is already \(\mathcal{L}^{\operatorname{mod}}\)-definable.  Since \(\CM^{\operatorname{mod}}\) is a reduct 
of \(\CM^{\operatorname{ind}}\), we see that 
\(\CM^{\operatorname{mod}}\) and \(\CM^{\operatorname{ind}}\) are 
bi-definable.
\end{proof}

Strong minimality of \(\CM^{\operatorname{ind}}\) follows from \Cref{prop:CM-mod-elementary} and
\Cref{prop:bi-definable}.

\begin{cor}
\label{cor:stability-CM-ind}
The theory of \(\CM^{\operatorname{ind}}\) is strongly minimal.  A fortiori 
this theory is stable without the finite cover property.
\end{cor}

In fact, by the effective Andr\'{e}--Oort theorem of Binyamini, the bi-definability of
\(\CM^{\operatorname{mod}}\) and \(\CM^{\operatorname{ind}}\) is itself effective.

\begin{prop}
\label{prop:effective-bi-define-CM-mod-ind}
There is a computable function \(\phi \mapsto \widehat{\phi}\) from 
the language of rings to \(\mathcal{L}^{\operatorname{mod}}\) so that for every
formula \(\phi\) of the language of rings, we have
\(\CM^{\operatorname{ind}} \models R_\phi \leftrightarrow \widehat{\phi}\).
\end{prop}
\begin{proof}
The theory of algebraically closed fields itself has an effective 
quantifier elimination algorithm.  So to compute 
\(\widehat{\phi}\) we may take \(\phi\) to be quantifier-free.
Our proof of \Cref{prop:bi-definable} shows that to compute 
\(\widehat{\phi}\) we need only have a way to compute from an algebraic 
variety \(X \subseteq \mathbb{A}^n_{\mathbb{Q}^{\operatorname{alg}}}\) the 
(possibly reducible) variety \(\overline{X(\CC) \cap \CM^n}^{\operatorname{Zariski}}\).

This may be accomplished using the degree bounds of~\cite[Theorem 1]{Bin-effective-AO};
as described in 
the footnote on~\cite[page 18]{Bin-effective-AO} we may carry out such a computation.  

Let us explain some of the details.  Using just the language of rings and 
our construction in the proof of \Cref{prop:bi-definable}, we need to 
give an effective method to transform the presentation of an algebraic 
variety \(X \subseteq \mathbb{A}^n_\mathbb{Q}\) into a presentation 
of \(\overline{X(\CC) \cap \CM^n}^{\operatorname{Zariski}}\).  We can effectively 
compute the irreducible components over \(X_{\mathbb{Q}^{\operatorname{alg}}}\), that is, the 
components of \(X\) regarded as an algebraic variety over the algebraic numbers.  
Computing \(\overline{Y(\CC) \cap \CM^n}^{\operatorname{Zariski}}\) for each such component 
gives a way to compute \(\overline{X(\CC) \cap \CM^n}^{\operatorname{Zariski}}\) itself by taking 
a finite union.  So, we will pass to the problem where \(X\) is an absolutely 
irreducible 
variety defined over some number field \(K\).  From~\cite[Theorem 1]{Bin-effective-AO}, 
we have a bound on the degree of any maximal special subvariety of \(X\) computed 
just from the degree of \(X\) and \([K:\mathbb{Q}]\).  From this we obtain a finite list of 
possible equations of the form \(\Phi_\ell(x_i,x_j) = 0\) to test.  The difficulty is 
determining which of the equations \(x_j = \xi\) with \(\xi \in \CM\) to test.  The third
bound in~\cite[Theorem 1]{Bin-effective-AO} says that the discriminant of such a \(\xi\) is bounded 
by a function of \(\deg(X)\), \([K:\mathbb{Q}]\), and a quantity \(\Delta^*\) computed 
as the discriminant of the ring of integers of a certain exceptional number field \(K^*\) 
associated with a theorem of Siegel and Tatuzawa.   Such a bound constrains \(\xi\) to a finite 
set, allowing us to complete our computation.  However, the reader might worry that we do not
really know the bound as it depends on \(\Delta^*\). Indeed, the computation of 
\(\Delta^*\) is an unresolved problem. Nevertheless, the number \(\Delta^*\) \emph{exists} 
and thus so too does our algorithm: for any number \(D\) we have 
an algorithm which we might call \(A_D\) in which 
we search for possible special subvarieties 
with constant coordinates \(\xi \in \CM\) with discriminant bounded by
\(c(D,\deg(X),[K:\mathbb{Q}])\), which is a computable function of \(D\),
\(\deg(X)\), and \([K:\mathbb{Q}]\).  Provided that \(D \geq \Delta^*\), 
algorithm \(A_D\) correctly produces all of the coordinates \(\xi\) we need 
to consider.
\end{proof}

Using a theorem of Casanovas and Ziegler on expansions of stable structures, 
we conclude that the theory of \(\fCM\) is tame.

\begin{thm}
\label{thm:stable-decidable-CM-ind}
The theory of \(\fCM\) is stable and decidable.
\end{thm}
\begin{proof}
The theory of the complex numbers in the language of rings is stable, 
the countable set \(\CM\) of CM-points is small (in the 
sense of~\cite{Casanovas-Ziegler-stable-theories-predicate}) in 
\(\CC\), and, by \Cref{cor:stability-CM-ind}, its induced structure
is strongly minimal, and thus stable
without the finite cover property.  Hence, by 
\cite[Theorem A]{Casanovas-Ziegler-stable-theories-predicate}, \(\fCM = (\CC,+,-,\cdot,0,1,\CM)\) 
has a stable theory.    From~\cite[Corollary 2.2]{Casanovas-Ziegler-stable-theories-predicate}
we see that \(\operatorname{Th}(\fCM)\) is axiomatized by \(\operatorname{Th}(\CC,+,-,\cdot,0,1) 
= \operatorname{ACF}_0\) and \(\operatorname{Th}(\CM^{\operatorname{ind}})\). Combining 
\Cref{prop:CM-mod-elementary} and
\Cref{prop:effective-bi-define-CM-mod-ind}, we see that 
 \(\operatorname{Th}(\CM^{\operatorname{ind}})\) is decidable.  Hence, \(\operatorname{Th}(\fCM)\) is 
 also decidable.
\end{proof}

\begin{Rk}
\label{rk:Pillay-decidability}
While the results of~\cite{Pillay-induced} are stated in a special case, namely, 
where the predicate by which 
the algebraically closed field is expanded names a finite-rank subgroup of a semiabelian variety, 
the proofs apply more generally and would suffice for our application in \Cref{thm:stable-decidable-CM-ind}.
\end{Rk}

\section{Questions on other structures interpretable in \texorpdfstring{$\CC(t)$}{C(t)}}
\label{sec:questions}

While \Cref{thm:stable-decidable-CM-ind} shows that the approach to proving the undecidability
of the theory of \(\CC(t)\) by
observing that the structure \(\fCM\) is interpretable in \(\CC(t)\) will not succeed, it does not
resolve the question of whether the other structures interpreted using similar methods are 
decidable.  In this section, we highlight some of those questions and discuss what we would need to know 
in order to resolve them.

Let us recall that \[\Isog := \{(x,y) \in \CC^2 : (\exists E)(\exists E')(\exists \psi:E \to E' \text{ an isogeny})~ j(E) = x 
\land j(E') = y \}\] is definable in \(\CC(t)\).  Using \(a \in \CC\) as a parameter, 
\(\Isog_a\), the set of \(j\)-invariants of elliptic curves isogenous to an elliptic curve with \(j\)-invariant 
\(a\), is thus also definable in \(\CC(t)\).   Using the Andr\'{e}-Pink-Zannier Conjecture as proven by
Richard and Yafaev~\cite[Theorem 1.3]{Rodolphe-Yafaev-apz} 
in place of the Andr\'{e}--Oort conjecture, we obtain an analogue of \Cref{prop:bi-definable}.

\begin{prop}
\label{prop:bi-define-isog-mod-ind}
For any complex number \(a \in \CC\) the structures \((\Isog_a^{\operatorname{ind}})_{\Isog_a}\) and \((\Isog_a^{\operatorname{mod}})_{\Isog_a}\)
are bi-definable.
\end{prop}

\begin{Rk}
\label{rk:parameters-in-bi-definition}
Note that in \Cref{prop:bi-define-isog-mod-ind} we need to allow for parameters from \(\Isog_a\).  Implicitly, we were 
already doing this for \(\CM^{\operatorname{mod}}\) in \Cref{prop:bi-definable} as every element of \(\CM\) is already definable by
a predicate there.
\end{Rk}

Unlike for \(\CM^{\operatorname{mod}}\), \(\Isog_a^{\operatorname{mod}}\) is not an elementary substructure of \(\CC^{\operatorname{mod}}\). However, we may still 
deduce quantifier elimination for \((\Isog_a^{\operatorname{mod}})_{\Isog_a}\) from what we know for \(\CC^{\operatorname{mod}}\).

\begin{prop}
\label{prop:qe-isog-a-mod}
For any complex number \(a \in \CC\)  
the theory of \((\Isog_a^{\operatorname{mod}})_{\Isog_a}\) admits quantifier elimination and is therefore 
strongly minimal.
\end{prop}
\begin{proof}
Working by induction and the usual reductions using irredundant representations of constructible sets, it suffices to 
show that if \(\phi\) is a quantifier-free formula in the free variables \(x_1, \ldots, x_{n+1}\) of the form 
\(\phi = R_X \land \neg \bigvee_{Y \in \mathcal{Y}} R_Y\) where \(X\) is a weakly special variety whose constant 
coordinates come from \(\Isog_a\) and each \(Y\) is a proper weakly special subvariety of \(X\), then 
\(\exists x_{n+1} \phi\) is equivalent to a quantifier-free formula.    By quantifier elimination for 
\(\CC^{\operatorname{mod}}\) there is a quantifier-free formula \(\theta\) so that 
\(\CC^{\operatorname{mod}} \models \theta \leftrightarrow \exists x_{n+1} \phi\).   
Express \(\theta\) in an irredundant disjunctive normal form as 
\(\theta = \bigvee_{Z \in \mathcal{Z}} \left( R_Z \land \neg \bigvee_{W \in \mathcal{W}_Z} R_W \right)\) where 
\(\mathcal{Z}\) is a finite set of weakly special varieties and for 
each \(Z \in \mathcal{Z}\), \(\mathcal{W}_Z\) is a finite set of proper weakly 
special subvarieties of \(Z\).  Let \(\mathcal{Z}' := \{ Z \in \mathcal{Z} :
\text{ no constant coordinate of } Z \text{ lies outside of } \Isog_a \} \).
Likewise, for \(Z \in \mathcal{Z}'\) let \(\mathcal{W}_Z'\) consist of those 
\(W \in \mathcal{W}_Z\) not having any constant coordinates outside of \(\Isog_a\).
Let \(\theta' := \bigvee_{Z \in \mathcal{Z}'} \left( R_Z \land \neg \bigvee_{W \in \mathcal{W}_Z'} R_W \right)\).
Noting that if \(V \subseteq \mathbb{A}^n_\CC\) has a constant coordinate outside 
of \(\Isog_a\), then \(V(\CC) \cap \Isog_a^n = \varnothing\), we see that 
\(\Isog_a \models \theta \leftrightarrow \theta'\).  

Suppose that for some \(b \in \Isog_a^n\) we have 
\(\Isog_a^{\operatorname{mod}} \models \exists x_{n+1} \phi(b,x_{n+1})\).
Then because \(\Isog_a\) is a substructure of \(\CC^{\operatorname{mod}}\), we have 
\(\CC^{\operatorname{mod}} \models \exists x_{n+1} \phi(b,x_{n+1}) \) 
as well.  From the characteristic property of 
\(\theta\), we then have that 
\(\CC^{\operatorname{mod}} \models \theta(b)\), which again gives
because \(\Isog_a\) is a substructure that 
\(\Isog_a^{\operatorname{mod}} \models \theta(b)\), which 
means that \(\Isog_a \models \theta'(b)\) as 
\(\theta\) and \(\theta'\) are equivalent in 
\(\Isog_a\). In the other direction, if 
\(\Isog_a \models \theta'(b)\), then 
\(\CC^{\operatorname{mod}} \models \exists x_{n+1} \phi(b,x_{n+1})\).
From the equations defining \(X\) we see that either 
\(X_b = \mathbb{A}^1_\CC\) (in which case the non-emptiness
of \(X_b(\CC) \smallsetminus \bigcup_{Y \in \mathcal{Y}} Y_b(\CC)\) implies that all but finitely many elements of 
\(\Isog_a\) could serve as witnesses of \(\exists x_{n+1} \phi\)),
the equation \(x_{n+1} = c\) holds on \(X\) for some \(c \in \Isog_a\),
so that \(X_b = \{ c \} \subseteq \Isog_a\), 
or \(\Phi_\ell(b_i,x) = 0\) holds on \(X_b\) for some 
\(1 \leq i \leq n\) and \(\ell \in \mathbb{Z}_+\) so that 
\(X_b(\CC)\) is contained in the isogeny class of 
\(b_i\), which is \(\Isog_a\).  In every case, we see 
that \(\Isog_a \models \exists x_{n+1} \phi\).
\end{proof}

\begin{cor}
\label{cor:stability-isogeny}
For any \(a \in \CC\) the theory of 
\((\CC,+,-,\cdot,0,1,\Isog_a)\) is stable.
\end{cor}
\begin{proof}
Apply \cite[Theorem A]{Casanovas-Ziegler-stable-theories-predicate}.
\end{proof}

To obtain decidability for the theory of 
\((\CC,+,-,\cdot,0,1,\Isog_a)\) we would need to know that the
function which takes a variety \(X \subseteq \mathbb{A}^n_{\mathbb{Q}(a)^{\operatorname{alg}}}\) to 
\(\overline{\Isog_a^n \cap X(\CC)}\) is computable. 
It seems plausible that the methods of~\cite{Bin-effective-AO}
could be adapted to produce such an algorithm.  In the 
case that \(a\) is transcendental, this can be deduced 
from~\cite{Fr-Sc-smj}.

The argument for \Cref{prop:qe-isog-a-mod} can be applied to 
finitely many isogeny orbits and the set of CM \(j\)-invariants 
at one time.  From this it follows that the structure 
\( (\CC,+,-,\cdot,0,1,\CM,(\Isog_a)_{a \in \CC})\) is 
stable and then 
if the expected analogues of 
\Cref{prop:effective-bi-define-CM-mod-ind} hold for 
\(\Isog_a\), then for any computable algebraically 
closed subfield \(K_0 \subseteq \CC\), the
structure \( (\CC,+,-,\cdot,0,1,\CM,(\Isog_a)_{a \in K_0})\)
would have a decidable theory.

We have seen that the relation \(\Isog\) itself is 
definable in \(\CC(t)\).  The set \(\Isog\) is not 
small relative to \(\CC\).  Thus, the main theorems 
of~\cite{Casanovas-Ziegler-stable-theories-predicate} do 
not immediately apply.  However, because the ambient theory 
of the complex numbers as a field is strongly minimal, 
the implications from stability without the finite cover 
property of the theory of the induced structure on 
\(\Isog\) to stability of \((\CC,+,-,\cdot,0,1,\Isog)\) and
likewise from decidability of the induced theory to 
decidability of the theory with a predicate still hold.  
Thus, we reduce to the following questions.

\begin{question}
\label{quest:stable-Isog}
Is the theory of \(\Isog^{\operatorname{ind}}\) stable?  Is its 
theory decidable?
\end{question}

It would seem that the Zilber--Pink conjecture should be 
relevant for resolving \Cref{quest:stable-Isog}.
However, even assuming that conjecture, the description 
of even the quantifier-free induced structure on 
\(\Isog\) is much more complicated than for \(\Isog^{\operatorname{mod}}\).
In~\cite{Pila-Scanlon-ZP} a differential algebraic analogue
\(\Xi\)
of \(\Isog\) is considered, namely the Kolchin closure of this
set with the constant points removed. The set \(\Xi\) 
has decidable, 
stable induced structure because it is a reduct of something
definable relative to the theory of differentially closed
fields of characteristic zero.   Moreover, 
\(\Xi\) satisfies a version of the Zilber--Pink conjecture.
While \(\Isog^{\operatorname{ind}}\) is not even a substructure
of \(\Xi^{\operatorname{ind}}\) (because \(\Xi\) is missing 
algebraic points), one might hope that their theories are
close enough to permit a description of the sets defined with 
quantifiers in \(\Isog^{\operatorname{ind}}\) through a comparison to 
the corresponding sets in \(\Xi^{\operatorname{ind}}\).  

There are some other sets definable in \(\CC(t)\) using 
variants of the method used to prove \Cref{cor:def-cm-isog}.  
For general abelian varieties \(A\) and \(B\) 
defined over \(\CC\), it is still the case that 
\(\operatorname{Hom}(A,B) = B(\CC(A))/B(\CC)\).  Thus,
if we have a uniformly definable family \(B \to S\) of 
abelian varieties and a uniformly definable family 
\(\{\CC(A)\}_{A \in F}\) of function fields of abelian 
varieties, there are first-order formulas specifying that the
rank of \(\operatorname{Hom}(A,B)\) is some fixed value.   

We do not know how to interpret function fields of higher-dimensional 
varieties in \(\CC(t)\).  If we 
could, then by a theorem of Eisentr\"{a}ger~\cite{Eis},
we would already know that \(\CC(t)\) is 
undecidable.  Instead, if \(X\) is a smooth projective 
curve of genus \(g \geq 1\) over \(\CC\) and we 
fix some point \(P \in X(\CC)\), then there is a map of 
algebraic varieties \(f:X \to J_X\) from the curve \(X\)
to its Jacobian, an abelian variety of dimension \(g\), having
the property that for any rational map \(h:X \dashrightarrow A\) 
from \(X\) to an abelian variety \(A\) with \(h(P) = 0_A\),
there is a unique map \(\overline{h}:J_X \to A\) with 
\(\overline{h} \circ f = h\).  Let us observe that  
if \(h:J_X \to A\) is a map of algebraic varieties from the 
Jacobian of \(X\) to some abelian variety \(A\), then 
\(h \circ f:X \to A\) is a map of algebraic varieties from 
\(X\) to \(A\).  As we have noted before, every rational 
map from an algebraic group to an abelian variety takes the 
form of a map of algebraic groups followed by translation
and every rational map from a smooth curve to an abelian variety is 
regular.  
Putting these observations together, we see that for any 
abelian variety \(A\) we have the identification 
\(\operatorname{Hom}(J_X,A) = A(\CC(X))/A(\CC)\). Hence, 
we can uniformly interpret the groups 
\(\operatorname{Hom}(J_X,A) \) of maps from Jacobians to 
abelian varieties.   I do not have useful conclusions about
what we might do with these groups to shed light on the 
(un)decidability of \(\CC(t)\). Some natural questions about
them follow.

\begin{question}
\label{quest:CM-higher-dimensional}
Let \(A \to S\) be a family of abelian varieties over 
the complex numbers.  Is the set 
\(\{a \in S(\CC) : A_a \text{ has 
 complex multiplication } \}\) definable in \(\CC(t)\)?  
\end{question}

I would guess that the answer to \Cref{quest:CM-higher-dimensional}  is positive and that the definition 
could be obtained by working with a suitable resolution 
of the abelian varieties in this family by Jacobians, but
to my knowledge, it remains an open problem whether such 
resolutions always exist.

\begin{question}
\label{quest:induced-higher}
Let \(A \to S\) be a family of abelian varieties over
the complex numbers and let \(X \to T\) be a family of 
smooth projective curves of some fixed genus over 
\(\CC\).  Fix a natural number \(r\).  Let 
\(H_r := \{ (t,s) \in (T \times S)(\CC) : \operatorname{rk} \operatorname{Hom}(J_{X_t},A_s) = r \} \).  
Is the induced structure on \(H_r\) stable?  Is it
decidable?
\end{question}

\bibliographystyle{siam}
\bibliography{dct}

@unpublished{AO-general,
author = {Pila, Jonathan and Shankar, Ananth N. and  Tsimerman, Jacob},
title = {Canonical Heights on {S}himura Varieties and the {A}ndr\'{e}-{O}ort Conjecture},
note = {with an appendix by H\'{e}l\`{e}ne Esnault  and
Michael Groechenig; \url{arXiv:2109.08788v4}},
year = {2024},
}

@unpublished{Sc-preprint-decidability,
    author = {Scanlon, Thomas},
    title = {Decidability of some complicated structures definable in \(\mathbb{C}(t)\)},
    note = {\url{arXiv:2508.17485v1}},
    Year = {2025},
}

@book {Silv-AEC,
    AUTHOR = {Silverman, Joseph H.},
     TITLE = {The arithmetic of elliptic curves},
    SERIES = {Graduate Texts in Mathematics},
    VOLUME = {106},
   EDITION = {Second},
 PUBLISHER = {Springer, Dordrecht},
      YEAR = {2009},
     PAGES = {xx+513},
      ISBN = {978-0-387-09493-9},
   MRCLASS = {11-02 (11G05 11G20 14H52 14K15)},
  MRNUMBER = {2514094},
MRREVIEWER = {Vasil\cprime  \={I}. Andr\={\i}\u{\i}chuk},
       DOI = {10.1007/978-0-387-09494-6},
       URL = {https://doi.org/10.1007/978-0-387-09494-6},
}

@book {Hartshorne-AG,
    AUTHOR = {Hartshorne, Robin},
     TITLE = {Algebraic geometry},
    SERIES = {Graduate Texts in Mathematics},
    VOLUME = {52},
 PUBLISHER = {Springer-Verlag, New York-Heidelberg},
      YEAR = {1977},
     PAGES = {xvi+496},
      ISBN = {0-387-90244-9},
   MRCLASS = {14-01},
  MRNUMBER = {463157},
MRREVIEWER = {Robert\ Speiser},
}

@article {Bin-effective-AO,
    AUTHOR = {Binyamini, Gal},
     TITLE = {Some effective estimates for {A}ndr\'{e}-{O}ort in {$Y(1)^n$}},
      NOTE = {With an appendix by Emmanuel Kowalski},
   JOURNAL = {J. Reine Angew. Math.},
  FJOURNAL = {Journal f\"{u}r die Reine und Angewandte Mathematik. [Crelle's
              Journal]},
    VOLUME = {767},
      YEAR = {2020},
     PAGES = {17--35},
      ISSN = {0075-4102},
   MRCLASS = {11G18 (11G50 14G05)},
  MRNUMBER = {4160301},
MRREVIEWER = {Davide Lombardo},
       DOI = {10.1515/crelle-2019-0028},
       URL = {https://doi.org/10.1515/crelle-2019-0028},
}

@article {Eis,
    AUTHOR = {Eisentr\"ager, Kirsten},
     TITLE = {Hilbert's tenth problem for function fields of varieties over
              {$\Bbb C$}},
   JOURNAL = {Int. Math. Res. Not.},
  FJOURNAL = {International Mathematics Research Notices},
      YEAR = {2004},
    NUMBER = {59},
     PAGES = {3191--3205},
      ISSN = {1073-7928,1687-0247},
   MRCLASS = {11U05 (03D35 11G05 11R58)},
  MRNUMBER = {2097039},
MRREVIEWER = {Zhi-Wei\ Sun},
       DOI = {10.1155/S1073792804141354},
       URL = {https://doi.org/10.1155/S1073792804141354},
}

@article {Rodolphe-Yafaev-apz,
    AUTHOR = {Richard, Rodolphe and Yafaev, Andrei},
     TITLE = {Generalised {A}ndr\'e-{P}ink-{Z}annier conjecture for
              {S}himura varieties of {A}belian type},
   JOURNAL = {Publ. Math. Inst. Hautes \'Etudes Sci.},
  FJOURNAL = {Publications Math\'ematiques. Institut de Hautes \'Etudes
              Scientifiques},
    VOLUME = {141},
      YEAR = {2025},
     PAGES = {249--331},
      ISSN = {0073-8301,1618-1913},
   MRCLASS = {11G18 (14G05 14G35)},
  MRNUMBER = {4916236},
       DOI = {10.1007/s10240-025-00154-4},
       URL = {https://doi.org/10.1007/s10240-025-00154-4},
}

@incollection {Zilber-dim,
    AUTHOR = {Zilber, Boris},
     TITLE = {Dimensions and homogeneity in mathematical structures},
 BOOKTITLE = {Connections between model theory and algebraic and analytic
              geometry},
    SERIES = {Quad. Mat.},
    VOLUME = {6},
     PAGES = {131--148},
 PUBLISHER = {Dept. Math., Seconda Univ. Napoli, Caserta},
      YEAR = {2000},
   MRCLASS = {03C45 (03C60)},
  MRNUMBER = {1930685},
MRREVIEWER = {Piotr L. Kowalski},
}

@incollection {Pink,
    AUTHOR = {Pink, Richard},
     TITLE = {A combination of the conjectures of {M}ordell-{L}ang and
              {A}ndr\'{e}-{O}ort},
 BOOKTITLE = {Geometric methods in algebra and number theory},
    SERIES = {Progr. Math.},
    VOLUME = {235},
     PAGES = {251--282},
 PUBLISHER = {Birkh\"{a}user Boston, Boston, MA},
      YEAR = {2005},
   MRCLASS = {11G10 (11G18 14G35)},
  MRNUMBER = {2166087},
MRREVIEWER = {Mihran Papikian},
       DOI = {10.1007/0-8176-4417-2\_11},
       URL = {https://doi.org/10.1007/0-8176-4417-2_11},
}

@book {Zan-unlikely,
    AUTHOR = {Zannier, Umberto},
     TITLE = {Some problems of unlikely intersections in arithmetic and
              geometry},
    SERIES = {Annals of Mathematics Studies},
    VOLUME = {181},
      NOTE = {With appendixes by David Masser},
 PUBLISHER = {Princeton University Press, Princeton, NJ},
      YEAR = {2012},
     PAGES = {xiv+160},
      ISBN = {978-0-691-15371-1},
   MRCLASS = {11G35 (11G15 11G18 11G50 14G05)},
  MRNUMBER = {2918151},
MRREVIEWER = {Yuri Bilu},
}

@book {ZP-panoramas,
    AUTHOR = {Habegger, P. and R\'{e}mond, G. and Scanlon, T. and Ullmo, E. and
              Yafaev, A.},
     TITLE = {Around the {Z}ilber-{P}ink conjecture/{A}utour de la
              conjecture de {Z}ilber-{P}ink},
    SERIES = {Panoramas et Synth\`eses [Panoramas and Syntheses]},
    VOLUME = {52},
      NOTE = {Papers corresponding to courses from the conference ``Etats de
              la Recherche'' held at CIRM, Marseille, May 2011},
 PUBLISHER = {Soci\'{e}t\'{e} Math\'{e}matique de France, Paris},
      YEAR = {2017},
     PAGES = {xiv+284},
      ISBN = {978-2-85629-856-5},
   MRCLASS = {11-06 (03-02 03C60 03C64 14G35)},
  MRNUMBER = {3752657},
}

@unpublished{milneAV,
author={Milne, James S.},
title={Abelian {V}arieties (v2.00)},
year={2008},
note={\url{https://www.jmilne.org/math/CourseNotes/AV.pdf}},
pages={vi+166}
}

@article {Casanovas-Ziegler-stable-theories-predicate,
    AUTHOR = {Casanovas, Enrique and Ziegler, Martin},
     TITLE = {Stable theories with a new predicate},
   JOURNAL = {J. Symbolic Logic},
  FJOURNAL = {The Journal of Symbolic Logic},
    VOLUME = {66},
      YEAR = {2001},
    NUMBER = {3},
     PAGES = {1127--1140},
      ISSN = {0022-4812,1943-5886},
   MRCLASS = {03C45},
  MRNUMBER = {1856732},
MRREVIEWER = {John\ T.\ Baldwin},
       DOI = {10.2307/2695097},
       URL = {https://doi.org/10.2307/2695097},
}

@book {Husemöller,
    AUTHOR = {Husem\"{o}ller, Dale},
     TITLE = {Elliptic curves},
    SERIES = {Graduate Texts in Mathematics},
    VOLUME = {111},
      NOTE = {With an appendix by Ruth Lawrence},
 PUBLISHER = {Springer-Verlag, New York},
      YEAR = {1987},
     PAGES = {xvi+350},
      ISBN = {0-387-96371-5},
   MRCLASS = {11G05 (11-01 11D25 14G25 14K07 14K15)},
  MRNUMBER = {868861},
MRREVIEWER = {David\ Grant},
       DOI = {10.1007/978-1-4757-5119-2},
       URL = {https://doi.org/10.1007/978-1-4757-5119-2},
}

@book {Knapp,
    AUTHOR = {Knapp, Anthony W.},
     TITLE = {Elliptic curves},
    SERIES = {Mathematical Notes},
    VOLUME = {40},
 PUBLISHER = {Princeton University Press, Princeton, NJ},
      YEAR = {1992},
     PAGES = {xvi+427},
      ISBN = {0-691-08559-5},
   MRCLASS = {11G05 (11-01 11F11 14H52)},
  MRNUMBER = {1193029},
MRREVIEWER = {Joseph\ H.\ Silverman},
}

@article {vdD-alg-int,
    AUTHOR = {van den Dries, Lou},
     TITLE = {Elimination theory for the ring of algebraic integers},
   JOURNAL = {J. Reine Angew. Math.},
  FJOURNAL = {Journal f\"ur die Reine und Angewandte Mathematik. [Crelle's
              Journal]},
    VOLUME = {388},
      YEAR = {1988},
     PAGES = {189--205},
      ISSN = {0075-4102,1435-5345},
   MRCLASS = {03C60 (11D72 12L05)},
  MRNUMBER = {944190},
MRREVIEWER = {\c S.\ A.\ Basarab},
       DOI = {10.1515/crll.1988.388.189},
       URL = {https://doi.org/10.1515/crll.1988.388.189},
}

@inproceedings {Rob-63,
    AUTHOR = {Robinson, Julia},
     TITLE = {The decision problem for fields},
 BOOKTITLE = {Theory of {M}odels ({P}roc. 1963 {I}nternat. {S}ympos.
              {B}erkeley)},
     PAGES = {299--311},
 PUBLISHER = {North-Holland, Amsterdam},
      YEAR = {1965},
   MRCLASS = {02.74 (10.80)},
  MRNUMBER = {0200163},
MRREVIEWER = {Walter Oberschelp},
}

@unpublished{Pila-Scanlon-ZP,
AUTHOR = {Pila, Jonathan and Scanlon, Thomas},
TITLE = {Effective transcendental {Z}ilber-{P}ink for variations of {H}odge structures},
YEAR = {2021},
NOTE = {\url{arXiv:2105.05845}},
}

@unpublished{Sc-AWS23,
AUTHOR = {Scanlon, Thomas},
TITLE = {Model theoretic origins and approaches to unlikely intersections problems \\ ({A}rizona {W}inter {S}chool 2023)},
YEAR = {2023},
NOTE = {\url{https://swc-math.github.io/aws/2023/2023ScanlonNotes.pdf}},
}

@article {Fr-Sc-smj,
     AUTHOR = {Freitag, James and Scanlon, Thomas},
      TITLE = {Strong minimality and the $j$-function},
    JOURNAL = {J. Eur. Math. Soc. (JEMS)},
   FJOURNAL = {Journal of the European Mathematical Society (JEMS)},
     VOLUME = {20},
       YEAR = {2018},
      PAGES = {119--136},
       ISSN = {1435-9855},
    MRCLASS = {03C60 (03C45\\12H05\\33E30)},
   MRNUMBER = {3743238},
 MRREVIEWER = {Cherlin, G.},
        DOI = {10.4171/JEMS/761},
        URL = {https://doi.org/10.4171/JEMS/761},
}

@article {UA-char,
    AUTHOR = {Ullmo, Emmanuel and Yafaev, Andrei},
     TITLE = {A characterization of special subvarieties},
   JOURNAL = {Mathematika},
  FJOURNAL = {Mathematika. A Journal of Pure and Applied Mathematics},
    VOLUME = {57},
      YEAR = {2011},
    NUMBER = {2},
     PAGES = {263--273},
      ISSN = {0025-5793,2041-7942},
   MRCLASS = {14G35},
  MRNUMBER = {2825237},
MRREVIEWER = {Lei\ Yang},
       DOI = {10.1112/S0025579311001628},
       URL = {https://doi.org/10.1112/S0025579311001628},
}

@article {ShWo,
    AUTHOR = {Shiga, Hironori and Wolfart, J\"urgen},
     TITLE = {Criteria for complex multiplication and transcendence
              properties of automorphic functions},
   JOURNAL = {J. Reine Angew. Math.},
  FJOURNAL = {Journal f\"ur die Reine und Angewandte Mathematik. [Crelle's
              Journal]},
    VOLUME = {463},
      YEAR = {1995},
     PAGES = {1--25},
      ISSN = {0075-4102,1435-5345},
   MRCLASS = {11G18 (11F46 11J86)},
  MRNUMBER = {1332905},
MRREVIEWER = {Michel\ Waldschmidt},
       DOI = {10.1515/crll.1995.463.1},
       URL = {https://doi.org/10.1515/crll.1995.463.1},
}

@incollection {Milne-AV,
    AUTHOR = {Milne, J. S.},
     TITLE = {Abelian varieties},
 BOOKTITLE = {Arithmetic geometry ({S}torrs, {C}onn., 1984)},
     PAGES = {103--150},
 PUBLISHER = {Springer, New York},
      YEAR = {1986},
      ISBN = {0-387-96311-1},
   MRCLASS = {14K05 (11G10)},
  MRNUMBER = {861974},
}

@article {Cohen,
    AUTHOR = {Cohen, Paula Beazley},
     TITLE = {Humbert surfaces and transcendence properties of automorphic
              functions},
      NOTE = {Symposium on Diophantine Problems (Boulder, CO, 1994)},
   JOURNAL = {Rocky Mountain J. Math.},
  FJOURNAL = {The Rocky Mountain Journal of Mathematics},
    VOLUME = {26},
      YEAR = {1996},
    NUMBER = {3},
     PAGES = {987--1001},
      ISSN = {0035-7596,1945-3795},
   MRCLASS = {11J89 (11F41 11F46 11G15 14K20)},
  MRNUMBER = {1428481},
MRREVIEWER = {Michel\ Waldschmidt},
       DOI = {10.1216/rmjm/1181072032},
       URL = {https://doi.org/10.1216/rmjm/1181072032},
}

@article {Schneider,
    AUTHOR = {Schneider, Theodor},
     TITLE = {Arithmetische {U}ntersuchungen elliptischer {I}ntegrale},
   JOURNAL = {Math. Ann.},
  FJOURNAL = {Mathematische Annalen},
    VOLUME = {113},
      YEAR = {1937},
    NUMBER = {1},
     PAGES = {1--13},
      ISSN = {0025-5831,1432-1807},
   MRCLASS = {99-04},
  MRNUMBER = {1513075},
       DOI = {10.1007/BF01571618},
       URL = {https://doi.org/10.1007/BF01571618},
}

@article {Pila-AO,
    AUTHOR = {Pila, Jonathan},
     TITLE = {O-minimality and the {A}ndr\'{e}-{O}ort conjecture for {$\Bbb
              C^n$}},
   JOURNAL = {Ann. of Math. (2)},
  FJOURNAL = {Annals of Mathematics. Second Series},
    VOLUME = {173},
      YEAR = {2011},
    NUMBER = {3},
     PAGES = {1779--1840},
      ISSN = {0003-486X},
   MRCLASS = {11G18 (03C64 11U09 14G35)},
  MRNUMBER = {2800724},
MRREVIEWER = {Alexandra Shlapentokh},
       DOI = {10.4007/annals.2011.173.3.11},
       URL = {https://doi.org/10.4007/annals.2011.173.3.11},
}

@incollection {Pillay-induced,
    AUTHOR = {Pillay, Anand},
     TITLE = {The model-theoretic content of {L}ang's conjecture},
 BOOKTITLE = {Model theory and algebraic geometry},
    SERIES = {Lecture Notes in Math.},
    VOLUME = {1696},
     PAGES = {101--106},
 PUBLISHER = {Springer, Berlin},
      YEAR = {1998},
   MRCLASS = {11U09 (03C45 11G10)},
  MRNUMBER = {1678531},
MRREVIEWER = {Ricardo Bianconi},
       DOI = {10.1007/978-3-540-68521-0\_6},
       URL = {https://doi.org/10.1007/978-3-540-68521-0_6},
}

@article {Zilber-mt-special,
    AUTHOR = {Zilber, Boris},
     TITLE = {Model theory of special subvarieties and {S}chanuel-type
              conjectures},
   JOURNAL = {Ann. Pure Appl. Logic},
  FJOURNAL = {Annals of Pure and Applied Logic},
    VOLUME = {167},
      YEAR = {2016},
    NUMBER = {10},
     PAGES = {1000--1028},
      ISSN = {0168-0072,1873-2461},
   MRCLASS = {03C45 (03C60 03C64 11J81 11J89 14G35)},
  MRNUMBER = {3522652},
MRREVIEWER = {G.\ Cherlin},
       DOI = {10.1016/j.apal.2015.02.002},
       URL = {https://doi.org/10.1016/j.apal.2015.02.002},
}

@incollection {KUY-bi-alg,
    AUTHOR = {Klingler, B. and Ullmo, E. and Yafaev, A.},
     TITLE = {Bi-algebraic geometry and the {A}ndr\'e-{O}ort conjecture},
 BOOKTITLE = {Algebraic geometry: {S}alt {L}ake {C}ity 2015},
    SERIES = {Proc. Sympos. Pure Math.},
    VOLUME = {97.2},
     PAGES = {319--359},
 PUBLISHER = {Amer. Math. Soc., Providence, RI},
      YEAR = {2018},
      ISBN = {978-1-4704-3578-3},
   MRCLASS = {11G18 (03C64 14G35)},
  MRNUMBER = {3821177},
MRREVIEWER = {Yunqing\ Tang},
       DOI = {10.1090/pspum/097.2/01709},
       URL = {https://doi.org/10.1090/pspum/097.2/01709},
}

\end{document}